\documentclass{amsart}

\usepackage{color,graphicx,amssymb,latexsym,amsfonts,txfonts,amsmath,amsthm}
\usepackage{pdfsync}
\usepackage{amscd}
\usepackage[all,cmtip]{xy}

\usepackage{hyperref}
\hypersetup{
    colorlinks=true,       
    linkcolor=blue,          
    citecolor=blue,        
    filecolor=blue,      
    urlcolor=blue           
}

\input epsf

\usepackage{graphics}
\usepackage{tikz}
\usetikzlibrary{arrows,decorations.markings,shapes.geometric}
\tikzstyle{point}=[circle, inner sep = 2pt, outer sep = 1pt, minimum size = 5pt, fill=black, draw=black]

\newcommand{\set}[1]{\left\{#1\right\}}

\newtheorem{theo}{Theorem}
\newtheorem{coro}{Corollary}

\newtheorem{lema}{Lemma}

\newtheorem{rema}{Remark}

\begin{document}

\title{Groups as automorphisms of dessins d'enfants}
\author[A. Ca\~{n}as]{Alejandro Ca\~{n}as}
\author[R.A. Hidalgo]{Rub\'en A. Hidalgo}
\author[F.J. Turiel]{Francisco Javier Turiel}
\author[A. Viruel]{Antonio Viruel}


\address{Departamento de \'Algebra, Geometr\'{\i}a y Topolog\'{\i}a, Universidad de M\'alaga, Espa\~na}
\email{alejandro.cm.95@uma.es, turiel@uma.es, viruel@uma.es}

\address{Departamento de Matem\'atica y Estad\'{\i}stica, Universidad de La Frontera. Temuco, Chile}
\email{ruben.hidalgo@ufrontera.cl}
\thanks{The first and third authors are partially supported by Ministerio de Ciencia e Innovaci\'on (Spain) grant PID2020-118452GB-I00. The second author is partially supported by Project Fondecyt 1190001. The first and fourth authors are partially supported by Ministerio de Ciencia e Innovaci\'on (Spain) grant PID2020-118753GB-I00. }

\begin{abstract}
It is known that every finite group can be represented as the full group of automorphisms of a suitable compact dessin d'enfant. In this paper, we give a constructive and easy proof that the same holds for any countable group by considering non-compact dessins. Moreover, we show that any tame action of a countable group is so realisable.
\end{abstract}

\maketitle

\section{Introduction}
In this paper, all surfaces are assumed to be orientable, connected, Hausdorff, second countable and without boundary. If $X$ is a surface, then we denote by ${\rm Hom}^{+}(X)$ the group of its orientation-preserving self-homeomorphisms. A topological action (or just an action) of an abstract group $G$ on a surface $X$ is an injective homomorphism $\theta:G \to {\rm Hom}^{+}(X)$ such that $\theta(G)$ acts tame on $X$ (see Section \ref{Sec:prelim}).

Every surface admits a Riemann surface structure and, if $S$ is a Riemann surface different from the Riemann sphere, the complex plane, the punctured plane, a torus, an annulus, the hyperbolic plane or the punctured hyperbolic plane, then  its group ${\rm Aut}(S)$ of conformal automorphisms is a countable group. Conversely, in  \cite{Greenberg}, Greenberg proved  that for every countable (finite or infinite) group $G$ there is a Riemann surface $S$ (which can be assumed to be compact if $G$ is finite) and an injective homomorphism $\theta:G \to {\rm Aut}(S)$. Moreover, $\theta$ may be assumed to be an isomorphism (by using maximal Fuchsian groups). A short proof of the previous fact has been given by Allcock in \cite{Allcock} (see also \cite{Wilkelman}). In the case that $S$ is hyperbolic, its universal cover is the hyperbolic plane ${\mathbb H}^{2}$. The hyperbolic metric of ${\mathbb H}^{2}$ induces a complete Riemannian metric of constant negative curvature on $S$ and it holds that its group ${\rm Isom}^{+}(S)$ of orientation-preserving isometries coincides with ${\rm Aut}(S)$.

A dessin d'enfant (or just a dessin) is a pair ${\mathcal D}=(X,{\mathcal G})$, where $X$ is a surface (either compact or not), ${\mathcal G} \subset X$ is a bipartite graph (vertices are either black or white and adjacent ones have different colors) where each vertex has finite degree and each connected component of $X \smallsetminus {\mathcal G}$, called a face, is homeomorphic to an open disc and bounded by a finite set of edges (an edge might be internal to a face).  Each dessin d'enfant on $X$ induces a unique (up to biholomorphisms) Riemann surface structure on it \cite{GiGo,JW}.
Dessins, on a compact surface $X$, were first introduced by Grothendieck in his {\it Esquisse d'un Programme} (1984) \cite{Gro} (we call these dessins as Grothendieck's dessins).
In this case, such a Riemann surface structure can be described by an algebraic curve defined over the field of algebraic numbers $\overline{\mathbb Q}$ and Grothendieck's idea was to use such a combinatorial tool to obtain information on the structure of the absolute Galois group ${\rm Gal}(\overline{\mathbb Q}/{\mathbb Q})$. Generalities on Grothendieck's dessins can be found, for instance, in \cite{GiGo,GiGo2,GZ,Gro,JW}.

The mapping class group associated to $X$ is defined to be ${\rm Map}(X)={\rm Hom}^{+}(X)/{\rm Hom}_{0}(X)$, where ${\rm Hom}_{0}(X)$ is the subgroup of homeomorphisms which are isotopic to the identity. Let  $\theta:{\rm Hom}^{+}(X) \to {\rm Map}(X)$ denote the quotient group projection. The automorphisms of the bipartite graph ${\mathcal G}$, denoted by  ${\rm Aut}_{bip}({\mathcal G})$, is the group  consisting of automorphisms of the graph that keep invariant black (respectively, white) vertices (called automorphisms of the bipartite graph).

Let $\widetilde{\rm Aut}({\mathcal D})$ be the subgroup of ${\rm Hom}^{+}(X)$ consisting of those
$\phi \in {\rm Hom}^{+}(X)$ keeping invariant ${\mathcal G}$ and inducing an element of ${\rm Aut}_{bip}({\mathcal G})$. Note that,
by performing any isotopy of $\phi \in \widetilde{\rm Aut}({\mathcal D})$ relative the vertices and keeping invariant the edges, we obtain the same automorphism of the bipartite graph; so this group is always uncountable.

The group ${\rm Aut}({\mathcal D}):=\theta(\widetilde{\rm Aut}({\mathcal D}))$ is called the group of automorphisms of the dessin ${\mathcal D}$ (this permits us to talk about the topological action of this group on $X$). There is a natural injective homomorphism $\rho:{\rm Aut}({\mathcal D}) \to {\rm Aut}_{bip}({\mathcal G})$, so this provides a natural copy of ${\rm Aut}({\mathcal D})$ inside ${\rm Aut}_{bip}({\mathcal G})$ (i.e., we can see ${\rm Aut}({\mathcal D})$ as those bipartite graph automorphisms of ${\mathcal G}$ that can be extended to an element of ${\rm Hom}^{+}(X)$). As previously said, the dessin ${\mathcal D}$ induces a (unique up to biholomorphisms) Riemann surface structure $S$ on $X$.
In this case, there is a subgroup of ${\rm Aut}(S)$ which is isomorphic (by $\theta$) to ${\rm Aut}({\mathcal D})$, that is, we may see
${\rm Aut}({\mathcal D}) \leq {\rm Aut}(S)={\rm Isom}^{+}(S)$. Now, as such a subgroup keeps invariant a non-empty discrete set of points (the black vertices),  it is always a countable group.

If ${\rm Aut}({\mathcal D})$ acts transitively on the set of edges of the graph, then the dessin is called regular (in this case, ${\rm Aut}({\mathcal D})$ is necessarily generated by two elements)  \cite{GiGo,JW}. If $X$ is compact, then ${\rm Aut}({\mathcal D})$ is necessarily finite. If $X$ is non-compact, then ${\rm Aut}({\mathcal D})$ might be infinite (but countable).

In \cite{H,GJ} it was observed that  every finite group
$G$ is isomorphic to the group of automorphisms of some Grothendieck's dessin (moreover, if $G$ is generated by two elements, then the dessin can be chosen to be regular).
In \cite{H}, it was also proved the existence for $G$ infinite, but assuming it was finitely generated.

Our first result is the following, which extends the above for any countable group. Moreover, we build a surface in which we realize any given group together with its subgroup lattice.

\begin{theo}\label{realizacion}
Let $G$ be a countable (finite or infinite) group. Then, there is a hyperbolic Riemann surface $X$ and a family of dessins $\set{(X,\mathcal{D}_H)}_{H\leq G}$ such that,
\begin{enumerate}
	\item [i)] ${\rm Isom}^{+}(X)\cong G$,
	\item [ii)] ${\rm Aut}(X,\mathcal{D}_H)=\set{h\in{\rm Isom}^+(X)\ |\ h(\mathcal{D}_H)=\mathcal{D}_H}\cong H$,
	\item [iii)] $H_1\leq H_2\ \Longrightarrow\ \mathcal{D}_{H_1}\leq\mathcal{D}_{H_2}$.
\end{enumerate}
Furthermore, if $G$ is finite, then $X$ can be chosen compact so obtaining a family of Grothendieck dessins.
\end{theo}

In \cite{H}, it was observed that, given a topological action of a finite group $G$ on a compact surface $X$, there is a Grothendieck's dessin ${\mathcal D}=(X,{\mathcal G})$ such that $G$  induces ${\rm Aut}({\mathcal D})$.
This result states that not only there is a dessin with a given group of automorphisms, but also that the given topological action is preserved by the automorphism group. In particular, the strong symmetric genus of a finite group is also the minimal genus action at the level of Grothendieck's dessins.
The proof of such result was done in terms of Riemann surfaces, Fuchsian groups and quasiconformal deformation theory. In this paper, we provide a simple  and short argument to the above realisation  that also works for any countable group.

\begin{theo}\label{main}
Let $X$ be a surface and fix a topologically tame action $\theta:G \to {\rm Hom}^{+}(X)$, where $G$ is a
countable (finite or infinite) group. Then there is a dessin ${\mathcal D}=(X,{\mathcal G})$ with ${\rm Aut}({\mathcal D})$ isomorphic to $G$ and induced by the topological action of $\theta(G)$.
\end{theo}

Let $G$ be any countable group and let $X$ be any surface of infinity type such that all of its ends are non-planar and its ends space is self-similar (see \cite{APV} for details). An example of such type of surfaces is, for instance, the Loch Ness monster  (the unique, up to orientation-preserving homeomorphisms, infinite genus surface with exactly one end).
In \cite{APV}, Aougab-Patel-Vlamis  proved that one can find a Riemann surface structure $S$ on $X$ such that ${\rm Aut}(S)$ is isomorphic to $G$. So, the above theorem asserts the following existence fact (as already stated before).

\begin{coro}
 Let $X$ be a surface of infinite type, where all of its ends are non-planar and its ends space is self-similar. If $G$ is a countable group, then there is a dessin ${\mathcal D}=(X,{\mathcal G})$ such that ${\rm Aut}({\mathcal D})$ isomorphic to $G$. In fact, such a dessin can be given in the Loch Ness monster.
\end{coro}

\section{Preliminaries}\label{Sec:prelim}

\subsection{Locally finite branched coverings}
A surjective continuous map $Q:X \to Y$, where $X$ and $Y$ are surfaces, is called a locally finite branched cover if:
\begin{enumerate}
\item the locus of branched values $B_{Q} \subset Y$ of $Q$ (it might be empty) is a discrete set;
\item $Q:X \smallsetminus Q^{-1}(B_{Q}) \to Y \smallsetminus B_{Q}$ is a covering map; and
\item each point $q \in B_{Q}$ has an open connected neighbourhood $U$ such that $Q^{-1}(U)$ consists of a collection $\{V_{j}\}_{j\in I}$ of pairwise disjoint connected open sets such that each of the restrictions $Q|_{V_j}:V_{j} \to U$  is a finite degree branched cover (i.e., it is topologically equivalent to a branched cover of the form $z \in {\mathbb D} \mapsto z^{d_{j}} \in {\mathbb D}$, where ${\mathbb D}$ denotes the unit disc).
\end{enumerate}

If $X$ is a compact surface, then a branched covering is always locally finite. But, in the case that $X$ is non-compact, a branched covering might not be locally finite.

\subsection{Topologically tame actions}
Let $G$ be a group, $X$ be a surface and $\theta:G \to {\rm Hom}^{+}(X)$ be an injective homomorphism. We say that the action of $\theta(G)$ is tame on $X$ if there is a Galois locally finite branched cover $Q:X \to Y$ with deck group $\theta(G)$ (in particular, $G$ is countable).

\subsection{Belyi pairs and dessins}
Let $S^{2}$ be the standard $2$-dimensional sphere and fix three different points $p_{0}, p_{1}, p_{\infty} \in S^{2}$. A dessin ${\mathcal D}=(X,{\mathcal G})$ defines a (unique up to isotopy) locally finite branched cover $Q:X \to S^{2}$ with $B_{Q} \subseteq \{p_{0},p_{1},p_{\infty}\}$, where $Q^{-1}(p_{0})$ (respectively, $Q^{-1}(p_{1})$) are the black (respectively, white) vertices of ${\mathcal G}$ and there is an arc $\delta=[p_{0},p_{1}] \subset S^{2} \smallsetminus \{p_{\infty}\}$, such that that $Q^{-1}(\delta)={\mathcal G}$. Moreover,
${\rm Aut}({\mathcal D})={\rm Aut}(X,Q)=\{ \phi \in {\rm Hom}^{+}(X): Q \circ \phi=Q\}$. Note that, for $X$ either of genus $g \geq 1$ or of infinite type, then $\#B_{Q}=3$. We say that the pair $(X,Q)$, as above, is a Belyi pair.
Conversely, given a  Belyi pair $(X,Q)$,
where $Q:X \to S^{2}$  is a locally finite  branched cover with $B_{Q} \subseteq \{p_{0},p_{1},p_{\infty}\}$, then it induces a dessin ${\mathcal D}$ as above.

\begin{rema}
Let $(X,Q)$ be a Belyi pair.
If $\varphi:S^{2} \to \widehat{\mathbb C}$ is an orientation-preserving homeomorphism, where $\widehat{\mathbb C}$ denotes the Riemann sphere, with $\varphi(B_{Q})\subseteq \{\infty,0,1\}$, then we may pull-back the Riemann surface structure of $\widehat{\mathbb C}$, under $\varphi \circ Q$, in order to provide a Riemann surface structure $S$ on $X$ such that $\varphi \circ Q:S \to \widehat{\mathbb C}$ is a non-constant meromorphic map with branch value set contained inside $\{\infty,0,1\}$. In this case, ${\rm Aut}({\mathcal D})$ is a subgroup of ${\rm Aut}(S)$ (the deck group of $\varphi \circ Q$). If $X$ is compact, then Belyi's theorem \cite{Belyi} asserts that the Riemann surface $S$ can be represented by an algebraic curve over $\overline{\mathbb Q}$. Giving an equivalent property in the case $X$ is non-compact is still an open problem.
\end{rema}

\subsection{Realisation of group actions on cubic simple graphs}
We refer to \cite{bollo} for the basic facts in Graph Theory.

\begin{theo}\label{3graph}
Let $G$ be a non-trivial countable group. There exists a connected cubic simple graph $\mathcal{G}$ and a proper $3$-labelling of its edges such that ${\rm Aut}_{LGraph}(\mathcal{G})\cong G$. Furthermore, there exist infinitely many such graphs, and if $G$ is finite, these graphs can be chosen to be finite as well.
\end{theo}

\begin{proof}
Let $S\subset G$ be a set of generators of $G$, we rule out the possibility of $e\in S$. Also, if $S$ is finite we will assume that $|S|>3$. If this is not the case, replace $S=\set{s_i}_{i=1}^{|S|}$ with the list $\set{s_1,s_1,s_1,s_1,s_2,\dots s_{|S|}}$.

Consider $\mathcal{C}=Cay(G,S)$ the associated Cayley graph. This is the directed graph with $G$ as the set of vertices and with set of edges $\sqcup_{i=1}^n\set{(g,gs_i)\in G\times G\mbox{ labeled }s_i}$. It is well known that ${\rm Aut}_{LGraph}(\mathcal{C})\cong G$.

We will now proceed to replace the vertices and edges of $\mathcal{C}$ with fixed configurations of vertices and edges allowing us to create the desired graph without modifying its automorphisms.

For each vertex $g$, there are $|S|$ edges arriving at it and $|S|$ edges coming from it. We replace $g$ with the following configuration of $2|S|$ vertices:

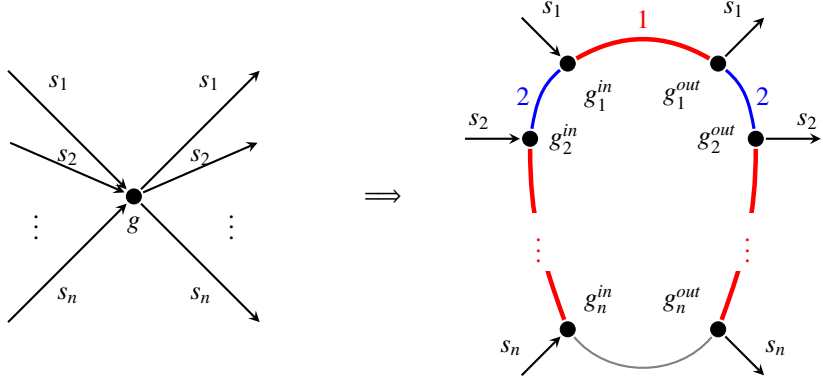
\begin{center}
\begin{tikzpicture}
	\tikzset{node distance = 1cm, auto}
	
	\node[style=point,color=white,label={[label distance=-0.2cm]0:$\Longrightarrow$}](C) {};	
	
	\node[style=point,left of = C,node distance=3cm,label={[label distance=0cm]270:$g$}](0) {};
	\node[style=point,color=white,above left of = 0,node distance=2.5cm](11) {};
	\node[style=point,color=white,below of = 11](12) {};
	\node[style=point,color=white,below left of = 0,node distance=2.5cm](13) {};
	\node[style=point,color=white,above right of = 0,node distance=2.5cm](21) {};
	\node[style=point,color=white,below of = 21](22) {};
	\node[style=point,color=white,below right of = 0,node distance=2.5cm](23) {};	
	
	\draw[->,thick,>=stealth] (11) to node[near start] {$s_1$} (0);
	\draw[->,thick,>=stealth] (12) to node[above=-0.1cm] {$s_2$} (0);
	\draw[->,thick,>=stealth] (13) to node[below=0.2cm] {$s_n$} (0);
	\draw[->,thick,>=stealth] (0) to node[near end] {$s_1$} (21);
	\draw[->,thick,>=stealth] (0) to node[above=-0.1cm] {$s_2$} (22);
	\draw[->,thick,>=stealth] (0) to node[below=0.2cm] {$s_n$} (23);
	
	\draw[->, draw=white] (12) to node[above=0.2cm,right=0.3cm] {$\vdots$} (13);
	\draw[->, draw=white] (22) to node[above=0.2cm,left=0.3cm] {$\vdots$} (23);
	
	\node[style=point,right of = 21,node distance=4cm,label={[label distance=0cm]-45:$g_1^{in}$}](011) {};
	\node[style=point,right of = 22,node distance=3.5cm,label={[label distance=0cm]0:$g_2^{in}$}](012) {};
	\node[style=point,right of = 23,node distance=4cm,label={[label distance=0cm]45:$g_n^{in}$}](013) {};
	\node[style=point,right of = 011,node distance=2cm,label={[label distance=0cm]180+45:$g_1^{out}$}](021) {};
	\node[style=point,right of = 012,node distance=3cm,label={[label distance=0cm]180:$g_2^{out}$}](022) {};
	\node[style=point,right of = 013,node distance=2cm,label={[label distance=0cm]180-45:$g_n^{out}$}](023) {};
	\node[style=point,color=white,above left of = 011](111) {};
	\node[style=point,color=white,left of = 012](112) {};
	\node[style=point,color=white,below left of = 013](113) {};
	\node[style=point,color=white,above right of = 021](121) {};
	\node[style=point,color=white,right of = 022](122) {};
	\node[style=point,color=white,below right of = 023](123) {};
	
	\draw[->,thick,>=stealth] (111) to node[near start] {$s_1$} (011);
	\draw[->,thick,>=stealth] (112) to node[near start] {$s_2$} (012);
	\draw[->,thick,>=stealth] (113) to node[near start] {$s_n$} (013);
	\draw[->,thick,>=stealth] (021) to node[near end] {$s_1$} (121);
	\draw[->,thick,>=stealth] (022) to node[near end] {$s_2$} (122);
	\draw[->,thick,>=stealth] (023) to node[near end] {$s_n$} (123);
	
	\draw[line width=0.6mm,color=red] (011) to[in=180-30,out=30] node {$1$} (021);
	\draw[very thick,color=blue] (011) to[in=80,out=180+40] node[left=0cm] {$2$} (012);
	\draw[very thick,color=blue] (021) to[in=180-80,out=-40] node[right=0cm] {$2$} (022);
	\draw[thick,color=gray] (013) to[in=180+50,out=-50] node {} (023);
	\draw[line width=0.6mm,color=red] (012) to[in=90+20,out=-90] node[above=-0.5cm,fill=white] {$\vdots$} (013);
	\draw[line width=0.6mm,color=red] (022) to[in=90-20,out=-90] node[above=-0.5cm,fill=white] {$\vdots$} (023);
\end{tikzpicture}
\end{center}

The edge connecting $g_n^{in}$ and $g_{n}^{out}$ only appears if $n=|S|<\infty$, and in that case its labeled $1$ or $2$ depending on the parity of $n$ (no consecutive edges can share labels).

Define $N=|S|+2$ if $S$ is finite, and $N=0$ if $|S|=\aleph_0$, and consider an edge labeled $s_i$. We replace this edge with the following construction:

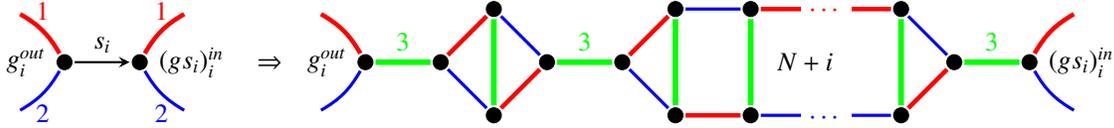
\begin{center}
\begin{tikzpicture}
	\tikzset{node distance = 1cm, auto}
	
	\node[style=point,color=white,label={[label distance=-0.2cm]0:$\Rightarrow$}](C) {};
	
	\node[style=point,left of = C,node distance=2.5cm,label={[label distance=0cm]180:$g_i^{out}$}](01) {};
	\node[style=point,right of = 01,node distance=1cm,label={[label distance=0cm]0:$(gs_i)_i^{in}$}](02) {};
	\node[style=point,color=white,above left of = 01](11) {};
	\node[style=point,color=white,below left of = 01](12) {};
	\node[style=point,color=white,above right of = 02](21) {};
	\node[style=point,color=white,below right of = 02](22) {};
	
	\draw[->,thick,>=stealth] (01) to node {$s_i$} (02);
	\draw[line width=0.6mm,color=red] (11) to[in=90+30,out=-30] node[above=0cm] {$1$} (01);
	\draw[very thick,color=blue] (12) to[in=180+60,out=30] node[below=0cm] {$2$} (01);
	\draw[line width=0.6mm,color=red] (21) to[in=60,out=180+30] node[above=0cm] {$1$} (02);
	\draw[very thick,color=blue] (22) to[in=-60,out=90+60] node[below=0cm] {$2$} (02);
	
	\node[style=point,right of = C,node distance=1.5cm,label={[label distance=0cm]180:$g_i^{out}$}](101) {};
	
	\node[style=point,right of = 101](A1) {};
	\node[style=point,above right of = A1](A21) {};
	\node[style=point,below right of = A1](A22) {};
	\node[style=point,below right of = A21](A3) {};
	\node[style=point,right of = A3](A4) {};
	\node[style=point,above right of = A4](A51) {};
	\node[style=point,below right of = A4](A52) {};
	\node[style=point,right of = A51](A61) {};
	\node[style=point,right of = A52](A62) {};
	\node[style=point,right of = A61,node distance = 2cm](A71) {};
	\node[style=point,right of = A62,node distance = 2cm](A72) {};
	\node[style=point,below right of = A71](A8) {};
	
	\node[style=point,right of = A8,label={[label distance=0cm]0:$(gs_i)_i^{in}$}](102) {};
	
	\node[style=point,color=white,above left of = 101](111) {};
	\node[style=point,color=white,below left of = 101](112) {};
	\node[style=point,color=white,above right of = 102](121) {};
	\node[style=point,color=white,below right of = 102](122) {};
	
	\draw[line width=0.6mm,color=red] (111) to[in=90+30,out=-30] node[above=0cm] {} (101);
	\draw[very thick,color=blue] (112) to[in=180+60,out=30] node[below=0cm] {} (101);
	\draw[line width=0.6mm,color=red] (121) to[in=60,out=180+30] node[above=0cm] {} (102);
	\draw[very thick,color=blue] (122) to[in=-60,out=90+60] node[below=0cm] {} (102);
	
	\draw[line width=0.75mm,color=green] (101) to node {$3$} (A1);
	\draw[line width=0.75mm,color=green] (A21) to node {} (A22);
	\draw[line width=0.75mm,color=green] (A3) to node {$3$} (A4);
	\draw[line width=0.75mm,color=green] (A51) to node {} (A52);
	\draw[line width=0.75mm,color=green] (A61) to node {} (A62);
	\draw[line width=0.75mm,color=green] (A71) to node {} (A72);
	\draw[line width=0.75mm,color=green] (A8) to node {$3$} (102);
	
	\node[style=point,color=white,right of = A4,label={[label distance=0.8cm]0:$N+i$}](n) {};
	
	\draw[line width=0.6mm,color=red] (A1) to node {} (A21);
	\draw[line width=0.6mm,color=red] (A22) to node {} (A3);
	\draw[line width=0.6mm,color=red] (A4) to node {} (A51);
	\draw[line width=0.6mm,color=red] (A52) to node {} (A62);
	\draw[line width=0.6mm,color=red] (A61) to node[fill=white,above=-0.25cm] {$\cdots$} (A71);
	\draw[line width=0.6mm,color=red] (A72) to node {} (A8);
	
	\draw[very thick,color=blue] (A1) to node {} (A22);
	\draw[very thick,color=blue] (A21) to node {} (A3);
	\draw[very thick,color=blue] (A4) to node {} (A52);
	\draw[very thick,color=blue] (A51) to node {} (A61);
	\draw[very thick,color=blue] (A62) to node[fill=white,above=-0.25cm] {$\cdots$} (A72);
	\draw[very thick,color=blue] (A71) to node {} (A8);
\end{tikzpicture}
\end{center}

The resulting graph $\mathcal{G}$ is a cubic connected graph with properly $3$-labeled edges. Let us see that, indeed ${\rm Aut}_{LGraph}(\mathcal{G})\cong G$. By construction, $3$-cycles only appear inside the configuration used to replace the edges of $\mathcal{C}$. We will always find four of these for each edge. Two of them share $2$ vertices and codifies the orientation of the original the edge in $\mathcal{C}$. The other two are joined by a ``ladder" of $(N+i)$ steps, consisting in $4$-cycles, which allows us to recover the label $s_i$. In order to recover the vertices of $\mathcal{C}$, remove all $3$-cycles from $\mathcal{G}$ and look at the connected components that remain. If $S$ is infinite, finite order connected components come from replacing edges, so each other connected component gives a vertex. If $S$ is finite, notice that connected components coming from edges must have order strictly greater than $2|S|$, so connected components of order $2|S|$ all come from a vertex. Finally note that the configuration used to replace edges is rigid, meaning it has no non-trivial automorphisms. This, together with the fact that $\mathcal{C}$ is completely determined by $\mathcal{G}$, gives ${\rm Aut}_{LGraph}(\mathcal{G})\cong {\rm Aut}_{LGraph}(\mathcal{C})\cong G$.
\end{proof}

\begin{lema}\label{action}
Let $\mathcal{G}$ be a cubic connected graph with properly $3$-labeled edges. Then, every $G\leq {\rm Aut}_{LGraph}(\mathcal{G})$ acts freely on the vertices of $\mathcal{G}$.
\end{lema}
\begin{proof}
Take $G\leq {\rm Aut}_{LGraph}{\mathcal{G}}$ and $g\in G$. Suppose $g(v)=v$ for some vertex $v$ of $\mathcal{G}$. For every other vertex $w$, consider a path from $v$ to $w$. Since $g$ is an automorphism of labeled graphs it must preserve the labelling of the path $v\rightarrow w$. Suppose this path is $v=v_0, v_1,\dots v_n=w$, and suppose the edges of $\mathcal{G}$ are labeled  with $\set{1,2,3}$. We know that the edge between $g(v_0)$ and $g(v_1)$ has the same label as the edge between $v_0$ and $v_1$. However, $g(v_0)=v_0$ and at any vertex of $\mathcal{G}$, there is only one adjacent edge for each label. This means the edge is the same and thus, $g(v_1)=v_1$. Clearly, by finite induction we arrive at $g(w)=w$. And since $w$ was arbitrary, $g=id$. Hence, $G$ acts freely on $\mathcal{G}$.
\end{proof}

\subsection{Certain hyperbolic right-angle hexagons}

Our proof of Theorem \ref{main} is built upon the existence of hyperbolic pants, that is hyperbolic hexagons, whose side lenghts are different enough. The next result shows these hexagons do exist.

\begin{lema}\label{hexagon}
There exists a hyperbolic right-angled hexagon with edges of length $a_1,b_1,a_2,b_2,a_3,b_3$, such that all lengths are different and $a_i<1<b_i$ for $i=1,2,3$.
\end{lema}
\begin{proof}

For any $l,\epsilon>0$, there exists a hyperbolic right-angled hexagon such that three non-consecutiveedges have lengths $l,l+\epsilon$ and $l-\epsilon$ \cite{ratcliffe}. Let $\alpha$, $\beta$, and $\gamma$ be the lengths of the edges opposing the edges of length $l$, $l+\epsilon$, and $l-\epsilon$ respectively. It is known that \cite{thurston},
$$\alpha=\cosh^{-1}\left(\frac{\cosh(l+\epsilon)\cosh(l-\epsilon)+\cosh(l)}{\sinh(l+\epsilon)\sinh(l-\epsilon)}\right).$$
By taking limits when $l\rightarrow\infty$ and $\epsilon\rightarrow0$, $\alpha$ goes to $0$, and this is also true for $\beta$ and $\gamma$. Therefore, by taking $l$ big enough and $\epsilon$ small enough we can have $0<\alpha,\beta,\gamma<1<l-\epsilon<l<l+\epsilon$. Finally, it can be also shown that \cite{thurston},
$$\frac{\sinh l}{\sinh \alpha}=\frac{\sinh (l+\epsilon)}{\sinh\beta}=\frac{\sinh(l-\epsilon)}{\sinh\gamma}.$$
And since $\sinh$ is a strictly increasing function, we deduce that $\alpha,\beta,\gamma$ are all different.
\end{proof}

\section{Proof of Theorem \ref{realizacion}}
\begin{proof}
The result is already known for the trivial group \cite{H,GJ} so we may assume $G$ is not trivial.

Choose $a_1,a_2,a_3,b_1,b_2,b_3$ as in Lemma \ref{hexagon}. Denote as $P$ the pair of hyperbolic pants obtained from glueing two copies of the hexagon given by the lemma along the edges of length $b_1,b_2,b_3$. The surface $P$ has three closed geodesics of length $2a_1,2a_2,2a_3$ as boundary. We will label them with $1,2,3$ respectively.
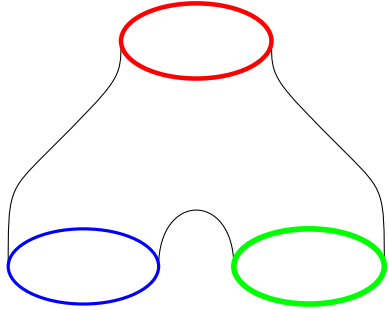
\begin{center}
\begin{tikzpicture}

\draw (-0.5,-3) .. controls (-0.5,-2) and (0.5,-2) .. (0.5,-3);
\draw (-1,0) .. controls (-1,-0.5) .. (-2,-1.5);
\draw (-2,-1.5) .. controls (-2.5,-2) .. (-2.5,-3);
\draw (1,0) .. controls (1,-0.5) .. (2,-1.5);
\draw (2,-1.5) .. controls (2.5,-2) .. (2.5,-3);

\draw[color=red,line width=0.6mm] (0,0) ellipse (1cm and 0.5cm);
\draw[color=blue,very thick] (-1.5,-3) ellipse (1cm and 0.5cm);
\draw[color=green,line width=0.75mm] (1.5,-3) ellipse (1cm and 0.5cm);

\end{tikzpicture}
\end{center}

Now, we build $X$ as follows. Consider $\mathcal{G}$ the graph associated to $G$ given by Theorem \ref{3graph}. For each vertex $v$ of $\mathcal{G}$, consider an isometric copy $P_v$ of $P$. And, for every edge labeled $i\in\set{1,2,3}$ between vertices $v$ and $w$, glue the pants $P_v$ and $P_w$ by isometrically identifying the geodesics with said label (thus without creating any twists).

Since the set of vertices of $\mathcal{G}$ is countable, we get an hyperbolic orientable surface $X$. If $G$ is finite, the set of vertices of $\mathcal{G}$ is also finite, so $X$ would be closed. Also note that $G$ permutes vertices in $\mathcal{G}$, that is it permutes pants, thus ${\rm Isom}^{+}(X)\cong G$. Indeed, an isometry maps closed geodesics into closed geodesics. The closed geodesics in $X$ are those from the boundary of each pair of pants, of lengths $2a_1,2a_2,2a_3<2$, and those that appear after joining at least three pants together, these ones have at least three geodesic segments of length $b_i>1$, so their length is at least $3$. Thus, the isometry must send the boundary of each pair of pants into another boundary. Therefore, isometries map pants into pants, allowing to translate them as automorphisms of $\mathcal{G}$. Also, the geodesics from the boundaries map to a geodesic of the same length, thus, it preserves the labelling of edges in $\mathcal{G}$. This proves $i)$.

Let us now build the dessin $\mathcal{D}_G$. For each vertex $v$ of $\mathcal{G}$, consider the following graph on $P_v$:
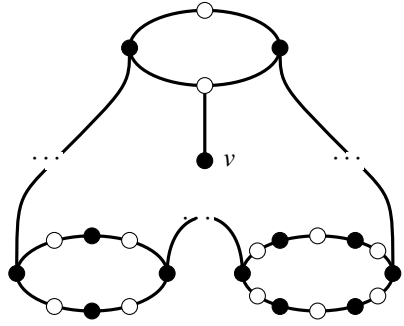
\begin{center}
\begin{tikzpicture}

\draw[very thick] (-0.5,-3) .. controls (-0.5,-2) and (0.5,-2) .. (0.5,-3);
\draw[very thick] (-1,0) .. controls (-1,-0.5) .. (-2,-1.5);
\draw[very thick] (-2,-1.5) .. controls (-2.5,-2) .. (-2.5,-3);
\draw[very thick] (1,0) .. controls (1,-0.5) .. (2,-1.5);
\draw[very thick] (2,-1.5) .. controls (2.5,-2) .. (2.5,-3);

\draw[very thick] (0,-0.5) -- (0,-1.5);

\draw[very thick] (0,0) ellipse (1cm and 0.5cm);
\draw[very thick] (-1.5,-3) ellipse (1cm and 0.5cm);
\draw[very thick] (1.5,-3) ellipse (1cm and 0.5cm);

\draw (0,-1.5) node[color=white,label={[label distance=0cm]0:$v$}](0) {};
\draw[fill=black] (0,-1.5) circle (3pt);
\draw[fill=white] (0,-0.5) circle (3pt);
\draw[fill=black] (-1,0) circle (3pt);
\draw[fill=black] (1,0) circle (3pt);
\draw[fill=white] (0,0.5) circle (3pt);

\draw[fill=black] (-2.5,-3) circle (3pt);
\draw[fill=black] (-0.5,-3) circle (3pt);
\draw[fill=black] (-1.5,-2.5) circle (3pt);
\draw[fill=white] (-1,-3.44) circle (3pt);
\draw[fill=white] (-2,-3.44) circle (3pt);
\draw[fill=black] (-1.5,-3.5) circle (3pt);
\draw[fill=white] (-1,-2.56) circle (3pt);
\draw[fill=white] (-2,-2.56) circle (3pt);

\draw[fill=black] (2.5,-3) circle (3pt);
\draw[fill=white] (2.3,-3.3) circle (3pt);
\draw[fill=white] (2.3,-2.7) circle (3pt);
\draw[fill=black] (2,-3.44) circle (3pt);
\draw[fill=black] (2,-2.56) circle (3pt);
\draw[fill=white] (1.5,-3.5) circle (3pt);
\draw[fill=white] (1.5,-2.5) circle (3pt);
\draw[fill=black] (1,-3.44) circle (3pt);
\draw[fill=black] (1,-2.56) circle (3pt);
\draw[fill=white] (0.7,-3.3) circle (3pt);
\draw[fill=white] (0.7,-2.7) circle (3pt);
\draw[fill=black] (0.5,-3) circle (3pt);

\draw (-2,-1.5) node[color=white,fill=white,label={[label distance=0cm]0:$ $}](1) {};
\draw (2,-1.5) node[color=white,fill=white,label={[label distance=0cm]0:$ $}](2) {};
\draw (0,-2.3) node[color=white,fill=white,label={[label distance=0cm]0:$ $}](3) {};

\draw (-2.55,-1.5) node[color=white,label={[label distance=0cm]0:$\cdots$}](1) {};
\draw (1.45,-1.5) node[color=white,label={[label distance=0cm]0:$\cdots$}](2) {};
\draw (-0.55,-2.3) node[color=white,label={[label distance=0cm]0:$\cdots$}](3) {};

\end{tikzpicture}
\end{center}

On the $3$ edges connecting two black vertices we place an alternating sequence of white-black vertices so that the resulting graph is bipartite. Each of these sequences has different length and all of them have length greater than the cycles on the boundary of $P$. Note that the surface obtained by removing this graph from $P_v$ is homeomorphic to the disjoint union of two open disks. By repeating the same pattern on all the pants, we obtain a graph $\mathcal{D}_G$ which makes $(X,\mathcal{D}_G)$ a dessin.

Finally, given a subgroup $H\leq G$ we define $\mathcal{D}_H$ as follows: Fix some $v_0\in\mathcal{G}$ and remove from $\mathcal{D}_G$ the vertices from the $G$-orbit of $v_0$ that are not in its $H$-orbit. This gives $\mathcal{D}_H\leq\mathcal{D}_G$ which makes $(X,\mathcal{D}_H)$ a dessin. And it is clear that this definition satisfies $iii)$.

Take an automorphism $f\in {\rm Aut}(X,\mathcal{D}_H)$. On each pair of pants $P_v$, $f$ preserves cycles on the boundary and thus, also the geodesics, so it must be induced by an element of ${\rm Isom}^{+}(X)$. This shows ${\rm Aut}(X,\mathcal{D}_H)=\set{h\in {\rm Isom}^{+}(X)\ |\ h(\mathcal{D}_H)=\mathcal{D}_H}$. Thus, $f\equiv h|_{\mathcal{D}_H}$ with $h\in {\rm Isom}^{+}(X)$ and $h(\mathcal{D}_H)=\mathcal{D}_H$. We know that $h$ maps pants into pants, but it has to also map $\mathcal{D}_H$ into itself. Hence, by construction, the pant $P_{v_0}$ will only map to another pant $P_w$ with $w\in Hv_0$. Lemma \ref{action} shows that this completely determines an element of the action of $H$ on $\mathcal{D}_H$, and thus, ${\rm Aut}(X,\mathcal{D}_H)\cong H$, proving $ii)$.

\end{proof}

\section{Proof of Theorem \ref{main}}
\begin{proof}
Let us consider a topologically tame action $\theta:G \to {\rm Hom}^{+}(X)$ and let $\pi:X \to Y$ be a Galois cover with deck group $\theta(G)$.
We may consider a dessin ${\mathcal D}_{0}=(Y,{\mathcal G})$ such that: (i) $B_{\pi}$ is contained in the set of white vertices of
${\mathcal G}$ and (ii) there is exactly one black vertice of degree one.
This dessin induces a locally finite branched cover $Q:Y \to S^{2}$ such that
\begin{enumerate}
\item the locus of branched values of $Q$ is $B_{Q} \subseteq \{p_{0},p_{1},p_{\infty}\} \subset S^{2}$;
\item $Q:Y \smallsetminus Q^{-1}(B_{Q}) \to S^{2} \smallsetminus B_{Q}$ is a covering map;
\item $Q^{-1}(p_{0})$ (respectively, $Q^{-1}(p_{1})$) are the black (respectively, white) vertices of ${\mathcal G}$, while there is a bijective correspondence between the points in $Q^{-1}(p_{\infty})$ and the connected components of $Y \smallsetminus {\mathcal G}$;
\item$Q^{-1}(p_{0}) \cap B_{\pi}=\emptyset$; and
\item in $Q^{-1}(p_{0})$ there is exactly one non-critical point $x_{0}$.
\end{enumerate}

The dessin ${\mathcal D}=(X,\pi^{-1}({\mathcal G}))$ satisfies that $\theta(G) \leq A={\rm Aut}({\mathcal D})$.
We claim that $A=\theta(G)$. Let us assume, by the contrary, that $\theta(G) \neq A$. In this case:
\begin{enumerate}
\item[(i)] there is a Galois (possible branched) cover $\pi_{A}:X \to R$, with deck group $A$,
\item[(ii)] there is a locally finite branched cover  $P:Y \to R$, degree at least two (it could be of infinite degree) such that $P \circ \pi=\pi_{A}$, and
\item[(iii)] there is a locally finite branched cover  $L:R \to S^{2}$ such that $L \circ P=Q$.

\end{enumerate}

Set $r_{0}=P(x_{0})$ and let $d \geq 1$ be the branched order of $\pi_{A}$ at $r_{0}$ (i.e., each point in the fiber $\pi_{A}^{-1}(r_{0})$ has $A$-stabilizer a cyclic group of order $d$.

As $B_{\pi} \cap Q^{-1}(p_{0})=\emptyset$, it follows that, at each point in the fiber of $P^{-1}(r_{0})$, the local degree of $P$ is the same $d$. But this means that, at each of these points, the local degree of $Q=L \circ P$ must be the same. This is a contradiction to condition (5) of $Q$.
\end{proof}

\begin{rema}[Another proof of Theorem \ref{realizacion}]\label{observa}
Let $H$ be a subgroup of $G$ and consider a locally finite branched Galois cover $U:X \to R$ induced by the action of $H$.
Take a locally finite branched cover $T:R \to Y$ such that $\pi=T \circ U$. Consider the preimage, under $T$, of the bipartite graph ${\mathcal G}$ (in the above proof) to obtain a bipartite graph $\widetilde{\mathcal G} \subset R$. In that case, all branch values of $U$ are contained in the white vertices and there is a non-empty collection of degree one black vertices. By removing  all the degree one black vertices (with the exception of one of them) and the adjacent edges, then we obtain a new bipartite graph ${\mathcal H} \subset \widetilde{\mathcal G}$. Such a bipartite graph satisfies the same conditions for $R$ and $U$ as ${\mathcal G}$ did for $Y$ and $\pi$. So, the same proof, permits to obtain a new dessin on $X$ (whose bipartite graph is contained in the bipartite graph $\pi^{-1}({\mathcal G})$) and whose automorphism group is $H$.
\end{rema}


\end{document}